\documentclass{elsarticle}

\usepackage{hyperref}

\journal{Discrete Mathematics}

\usepackage{latexsym}
\usepackage{graphics}
\usepackage{amssymb}
\usepackage{amsmath}
\usepackage{amscd}
\usepackage{amsthm}
\usepackage{tikz}

\newtheorem{theorem}{Theorem}

\newtheorem{cor}[theorem]{Corollary}
\newtheorem{prop}[theorem]{Proposition}
\newtheorem*{thm1}{Theorem 1}

\theoremstyle{definition}
\newtheorem{definition}[theorem]{Definition}

\theoremstyle{remark}

\begin{document}

\begin{frontmatter}

\title{Maximal harmonic group actions on finite graphs}

\author{Scott Corry\corref{mycorrespondingauthor}}
\cortext[mycorrespondingauthor]{Corresponding author}
\ead{scott.corry@lawrence.edu}
\address{Department of Mathematics, Lawrence University, 711 E. Boldt Way -- SPC 24, Appleton, WI 54911, USA}

\begin{abstract} This paper studies groups of maximal size acting harmonically on a finite graph.
Our main result states that these maximal graph groups are exactly the finite quotients of the modular group $\Gamma=\left<x,y \ | \ x^2=y^3=1\right>$ of size at least 6. This characterization may be viewed as a discrete analogue of the description of Hurwitz groups as finite quotients of the $(2,3,7)$-triangle group in the context of holomorphic group actions on Riemann surfaces. In fact, as an immediate consequence of our result, every Hurwitz group is a maximal graph group, and the final section of the paper establishes a direct connection between maximal graphs and Hurwitz surfaces via the theory of combinatorial maps.
\end{abstract}


\end{frontmatter}


\section{Introduction} Many recent papers have explored analogies between Riemann surfaces and finite graphs (e.g. \cite{Bak},\cite{BNRR},\cite{BN},\cite{CapTM},\cite{Cap},\cite{CapViv},\cite{GenusBnds},\cite{HarmGal},\cite{GlaMer}). Inspired by the Accola-Maclachlan \cite{Acc}, \cite{Mac} and Hurwitz \cite{Hur} genus bounds for holomorphic group actions on compact Riemann surfaces, we introduced harmonic group actions on finite graphs in \cite{GenusBnds}, and established sharp linear genus bounds for the maximal size of such actions. As noted in the introduction to \cite{GenusBnds}, it is an interesting problem to characterize the groups and graphs that achieve the upper bound $6(g-1)$. Such maximal groups and graphs may be viewed as graph-theoretic analogues of Hurwitz groups and surfaces---those compact Riemann surfaces $\mathcal{S}$ of genus $g\ge 2$ such that $\textrm{Aut}(\mathcal{S})$ has maximal size $84(g-1)$. This paper provides a description of the maximal graphs and groups (Theorem~\ref{Main} and Proposition~\ref{DM}), while also establishing connections between the recent theory of harmonic group actions and the well-studied topics of trivalent symmetric graphs and regular combinatorial maps.

The investigation of Hurwitz groups has been a rich and active area of research, and much is known about their classification including a complete analysis of the 26 sporadic simple groups: 12 of them (including the Monster!) are Hurwitz, while the other 14 are not (see \cite{Con}, \cite{ConU} for an overview). One starting point for work on Hurwitz groups is the following generation result: a finite group $G$ is a Hurwitz group if and only if it is a non-trivial quotient of the (2,3,7)-triangle group $\Delta$ with presentation
$$
\Delta = \left<x,y \ | \ x^2=y^3=(xy)^7=1\right>.
$$
That is: the Hurwitz groups are exactly the finite groups generated by an element of order 2 and an element of order 3 such that their product has order 7. The connection between the abstract group $\Delta$ and Hurwitz groups comes from the fact that Hurwitz surfaces arise as branched covers of the thrice-punctured Riemann sphere with special ramification. Such covers are nicely classified by the fundamental group of the punctured sphere, which is a free group on two generators.

The main result of this paper is an analogous generation result for \emph{maximal graph groups} -- those finite groups of size $6(g-1)$ that act harmonically on a finite graph of genus $g\ge 2$:

\begin{theorem}\label{Main}
A finite group $G$ is a maximal graph group if and only if $|G|\ge 6$ and $G$ is a quotient of the modular group $\Gamma$ with presentation 
$$
\Gamma=\left<x,y \ | \ x^2=y^3=1\right>.
$$
\end{theorem}

\noindent
That is: the maximal graph groups are exactly the finite groups generated by an element of order 2 and an element of order 3. As an immediate corollary, we have:

\begin{cor}
Every Hurwitz group is a maximal graph group.
\end{cor}

As in the case of Hurwitz groups, the connection between the modular group $\Gamma$ and maximal graph groups comes from the fact that maximal graphs occur as harmonic branched covers of trees (genus 0 graphs) with special ramification (Proposition \ref{BranchLoc}). In order to classify such covers in general, we developed a harmonic Galois theory for finite graphs in \cite{HarmGal}, and the resulting concrete description of harmonic branched covers is the main tool used in the proof of Theorem~\ref{Main}, which we present in section~\ref{Proof}. The proof of Theorem~\ref{Main} leads immediately to Proposition~\ref{DM}, which provides a close connection between maximal graphs and trivalent symmetric graphs of type 1' studied by Djokovi\'c and Miller in \cite{DM}.

The relation between Riemann surfaces and finite graphs explored in this paper is largely analogical, rather than arising from a precise correspondence. However, there are a variety of direct connections between Riemann surfaces (and more generally algebraic curves) and finite graphs (see e.g. \cite{Bak},\cite{BM},\cite{CapTM},\cite{Cap},\cite{Jones}). Of particular interest for us is a portion of the well-established theory of combinatorial maps, whereby the specification of a cyclic ordering of the edges incident to each vertex of a finite graph determines a 2-cell embedding of the graph in a compact Riemann surface. In the final section of this paper, we show that our theory meshes well with this construction in the following sense: if $G$ is a maximal graph group, then (by Proposition~\ref{DM}) $G$ acts maximally on a trivalent graph $Y_0$. Moreover, the $G$-action endows $Y_0$  with a cyclic ordering of the three edges at  each vertex, and $G$ acts as a group of holomorphic automorphisms of the corresponding Riemann surface. Moreover, if $G$ is actually a Hurwitz group, then the resulting surface is a Hurwitz surface with automorphism group $G$. 

\section{Harmonic Group Actions}\label{HarmGroupActions}

In this section, we briefly review some of the definitions and results from \cite{BN}, \cite{GenusBnds}, and \cite{HarmGal}. To begin, by a \emph{graph} we mean a finite multi-graph without loop edges: two vertices may be connected by multiple edges, but no vertex has an edge to itself. We denote the (finite) vertex-set of a graph $X$ by $V(X)$, and the (finite) edge-set by $E(X)$. For a vertex $x\in V(X)$, we write $x(1)$ for the subgraph of $X$ induced by the edges incident to $x$:
\begin{align*}
V(x(1))&:=\{x\}\cup\{w\in V(X) \ | \ w \textrm{ is adjacent to $x$}\}\\ 
E(x(1))&:=\{e\in E(X) \ | \ e \textrm{ is incident to $x$}\}.
\end{align*}
The \emph{genus}\footnote{In graph theory, the term ``genus'' usually refers to the minimal genus of an orientable surface into which a graph may be embedded, while the first Betti number of the graph is called the \emph{cyclomatic number}. But following \cite{BNRR}, we will refer to the quantity $g(X)$ as the genus, because in the theory of divisors on the graph $X$, it plays a role analogous to the genus of a Riemann surface.} of a connected graph $X$ is the rank of its first Betti homology group: $g(X):=|E(X)|-|V(X)|+1$.

\begin{definition} A \emph{morphism of graphs} $\phi:Y\rightarrow X$ is a function $\phi:V(Y)\cup E(Y)\rightarrow V(X)\cup E(X)$ mapping vertices to vertices and such that for each edge $e\in E(Y)$ with endpoints $y_1\ne y_2$, either $\phi(e)\in E(X)$ has endpoints $\phi(y_1)\ne \phi(y_2)$, or $\phi(e)=\phi(y_1)=\phi(y_2)\in V(X)$. In the latter case, we say that the edge $e$ is \emph{$\phi$-vertical}. The morphism $\phi$ is \emph{degenerate} at $y\in V(Y)$ if $\phi(y(1))=\{\phi(y)\}$, i.e. if $\phi$ collapses a neighborhood of $y$ to a vertex of $X$. The morphism $\phi$ is \emph{harmonic} if for all vertices $y\in V(Y)$, the quantity $|\phi^{-1}(e')\cap y(1)|$ is independent of the choice of edge $e'\in E(\phi(y)(1))$. 
\end{definition}

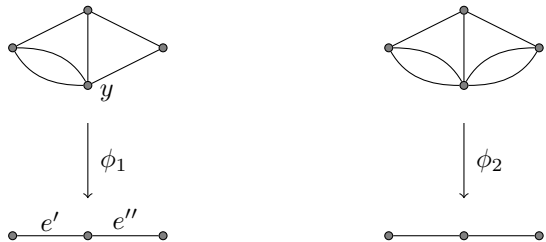
\begin{figure}[h]
\centering
\begin{tikzpicture}
\tikzstyle{every node}=[circle, draw, fill=black!50, inner sep=0pt, minimum width=3pt]

\node (1) at  (-1,1)  {};
\node (3) at  (1,1) {};
\node (2) at (0,1.5) {};
\node (4) at  (0,.5)  {};
\coordinate [label=right : $y$] (y) at (0.1,.4);

\draw [-] (1) to (2);
\draw [-] (2) to (3);
\draw [-] (2) to (4);
\draw [-] (4) to (3);
\draw[-] (1) to [out=0, in=120] (4);
\draw[-] (1) to [out=300, in=180] (4);

\draw[->] (0,0) -- (0,-1);
\coordinate [label=right : $\phi_1$] (phi1) at (0.1,-.5);

\node (5) at (-1,-1.5) {};
\coordinate [label=above : $e'$] (e') at (-0.5,-1.5);
\node (6) at (0,-1.5) {};
\coordinate [label=above : $e''$] (e'') at (0.5,-1.5);
\node (7) at (1,-1.5) {};

\draw[-] (5) to (6);
\draw[-] (6) to (7);

\node (8) at  (4,1)  {};
\node (10) at  (6,1) {};
\node (9) at (5,1.5) {};
\node (11) at  (5,.5)  {};

\draw [-] (8) to (9);
\draw [-] (9) to (10);
\draw [-] (9) to (11);
\draw[-] (8) to [out=0, in=120] (11);
\draw[-] (8) to [out=300, in=180] (11);
\draw[-] (10) to [out=180, in=60] (11);
\draw[-] (10) to [out=240, in=0] (11);

\draw[->] (5,0) -- (5,-1);
\coordinate [label=right : $\phi_2$] (phi2) at (5.1,-.5);

\node (12) at (4,-1.5) {};
\node (13) at (5,-1.5) {};
\node (14) at (6,-1.5) {};

\draw[-] (12) to (13);
\draw[-] (13) to (14);

\end{tikzpicture}
\caption{Each morphism is given by vertical projection. $\phi_1$ is not harmonic at $y$, because the edge $e'$ has two pre-images incident to $y$, while the edge $e''$ has only one. The morphism $\phi_2$ is harmonic.}
\end{figure}

\begin{definition}\label{deg}
Let $\phi:Y\rightarrow X$ be a harmonic morphism between graphs, with $X$ connected. If $|V(X)|>1$ (i.e. if $X$ is not the point graph $\star$), then
the \emph{degree} of the harmonic morphism $\phi$ is the number of pre-images in $Y$ of any edge of $X$ (this is well-defined by \cite{BN}, Lemma 2.4). If $X=\star$ is the point graph, then the \emph{degree} of $\phi$ is defined to be $|V(Y)|$, the number of vertices of $Y$.
\end{definition}

\begin{definition}
Suppose that $G\le\textrm{Aut}(Y)$ is a (necessarily finite) group of automorphisms of the graph $Y$, so that we have a left action $G\times Y\rightarrow Y$ of $G$ on $Y$. We say that $(G,Y)$ is a \emph{faithful group action} if the stabilizer of each connected component of $Y$ acts faithfully on that component. Note that this condition is automatic if $Y$ is connected.
\end{definition} 

Given a faithful group action $(G,Y)$, we denote by $G\backslash Y$ the quotient graph with vertex-set $V(G\backslash Y)=G\backslash V(Y)$, and edge-set 
$$
E(G\backslash Y)=G\backslash E(Y) - \{Ge \ | \ \textrm{$e$ has endpoints $y_1,y_2$ and $Gy_1=Gy_2$}\}.
$$
Thus, the vertices and edges of $G\backslash Y$ are the left $G$-orbits of the vertices and edges of $Y$, with any loop edges removed. There is a natural morphism $\phi_G:Y\rightarrow G\backslash Y$ sending each vertex and edge to its $G$-orbit, and such that edges of $Y$ with endpoints in the same $G$-orbit are $\phi_G$-vertical. As demonstrated in Figure~\ref{NonHarm}, the quotient morphism $\phi_G$ is not necessarily harmonic, which motivates the following definition.
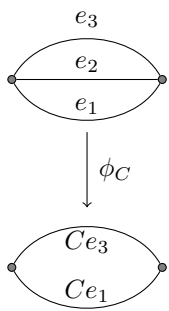
\begin{figure}[h]
\centering
\begin{tikzpicture}
\tikzstyle{every node}=[circle, draw, fill=black!50, inner sep=0pt, minimum width=3pt]

\node (1) at  (-1,0)  {};
\node (2) at  (1,0) {};

\draw[-] (2) to [out=240, in=300] (1);
\coordinate [label=above : $e_1$] (e1) at (0,-.55);
\draw[-] (2) to (1);
\coordinate [label=above : $e_2$] (e2) at (0,0);
\draw[-] (2) to [out=120, in=60] (1);
\coordinate [label=above : $e_3$] (e3) at (0,.6);

\draw[->] (0,-.7) -- (0,-1.7);
\coordinate [label=right : $\phi_C$] (phi1) at (0.1,-1.2);

\node (3) at  (-1,-2.5)  {};
\node (4) at  (1,-2.5) {};

\draw[-] (4) to [out=240, in=300] (3);
\coordinate [label=above : $Ce_1$] (e1) at (0,-3.15);
\draw[-] (4) to [out=120, in=60] (3);
\coordinate [label=above : $Ce_3$] (Ce3) at (0,-2.5);

\end{tikzpicture}

\caption{The cyclic group $C=\mathbb{Z}/2\mathbb{Z}$ acts faithfully on the upper graph by interchanging the edges $e_1$ and $e_2$ while fixing the edge $e_3$. The quotient morphism $\phi_C$ is not harmonic, because the edge of the quotient graph corresponding to the orbit $Ce_3$ has only one pre-image (the edge $e_3$), while the edge corresponding to $Ce_1$ has two preimages ($e_1$ and $e_2$). }
\label{NonHarm}

\end{figure}

\begin{definition}\label{HA} Suppose that $(G,Y)$ is a faithful group action. Then $(G,Y)$ is a \emph{harmonic group action} if for all subgroups $H<G$, the quotient morphism $\phi_H:Y\rightarrow H\backslash Y$ is harmonic.
\end{definition}

The condition in Definition~\ref{HA} is quite restrictive, but the following proposition provides a simple criterion for harmonicity:
\begin{prop}[\cite{HarmGal} Prop. 2.7; \cite{GenusBnds} Prop. 2.5]\label{Crit}
Suppose that $(G,Y)$ is a faithful group action. Then $(G,Y)$ is a harmonic group action if and only if for every vertex $y\in V(Y)$, the stabilizer subgroup $I_y\le G$ acts freely on the edge-set $E(y(1))$. Equivalently, $(G,Y)$ is harmonic if and only if (after assigning an arbitrary direction to each edge of $Y$), the stabilizer subgroup of every directed edge is trivial.
\end{prop}

By Proposition~\ref{Crit}, if $(G,Y)$ is a harmonic group action, then no directed edge of $Y$ is fixed by a non-identity element of $G$, which implies that the stabilizers of (non-directed) edges of $Y$ are either trivial or of order 2. That is: if the edge $e\in E(Y)$ is sent to itself by an element $\tau\in G$, then $\tau$ is an involution that switches the two endpoints of $e$. We refer to such an edge $e$ as \emph{flipped}, and if there are no flipped edges, then we say that the harmonic group action $(G,Y)$ is \emph{unflipped}. As explained in section 2 of \cite{HarmGal}, any harmonic group action $(G,Y)$ has a unique \emph{unflipped model}, obtained by replacing each flipped edge $e$ with a pair of edges $e,e'$ that are interchanged by the involution $\tau$.

\subsection{Genus Bounds}

In \cite{GenusBnds}, we established graph-analogues of the linear genus bounds for the maximal size of the automorphism group of a compact Riemann surface of genus $g\ge 2$. The situation for surfaces, as developed by Hurwitz \cite{Hur}, Accola \cite{Acc}, and Maclachlan \cite{Mac}, goes as follows. For each $g\ge 2$, define
$$
N(g):= \textrm{max}\{|\textrm{Aut}(\mathcal{S})| \ | \ \mathcal{S} \textrm{ is a compact Riemann surface of genus $g$}\}.
$$
Then $8(g+1)\le N(g)\le 84(g-1)$, and both of these bounds are sharp in the sense that the extreme values $8(g+1)$ and $84(g-1)$ are each attained infinitely often. $\mathcal{S}$ is called a \emph{Hurwitz surface} if it attains the upper bound: $|\textrm{Aut}(\mathcal{S})|=84(g(\mathcal{S})-1)$. A finite group $G$ is called a \emph{Hurwitz group} if there exists a Hurwitz surface $\mathcal{S}$ with automorphism group isomorphic to $G$. The smallest Hurwitz group is $PSL_2(\mathbb{F}_7)$ which occurs in genus 3 as the automorphism group of Klein's quartic curve defined in homogeneous coordinates by the equation
$$
x^3y+y^3z+z^3x=0.
$$

We now describe graph-theoretic versions of these results from \cite{GenusBnds}. For each $g\ge 2$, define
$$
M(g):= \textrm{max}\{|G| \ | \ G \textrm{ acts harmonically on a connected graph of genus $g$}\}.
$$
Then $4(g-1)\le M(g)\le 6(g-1)$, and these bounds are sharp in the sense that the extreme values $4(g-1)$ and $6(g-1)$ are each attained infinitely often. Moreover, unlike the case of Riemann surfaces, these two extremes are actually the \emph{only} values taken by the function $M(g)$. A connected graph $Y$ is called a \emph{maximal graph} if it attains the upper bound, i.e. if there exists a finite group $G$ acting harmonically on $Y$ with $|G|=6(g(Y)-1)$. In this case we call $G$ a \emph{maximal graph group} and say that $G$ acts \emph{maximally} on $Y$. The smallest maximal graph group occurs already in genus 2. In fact both groups of order $6=6(2-1)$ are maximal graph groups: the symmetric group $\mathfrak{S}_3$ and the cyclic group $\mathbb{Z}/6\mathbb{Z}$ each act maximally on the genus 2 graph consisting of 2 vertices connected by 3 edges (see Figure~\ref{genus 2}).

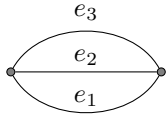
\begin{figure}[h]
\centering
\begin{tikzpicture}
\tikzstyle{every node}=[circle, draw, fill=black!50, inner sep=0pt, minimum width=3pt]

\node (1) at  (-1,0)  {};
\node (2) at  (1,0) {};

\draw[-] (2) to [out=240, in=300] (1);
\coordinate [label=above : $e_1$] (e1) at (0,-.55);
\draw[-] (2) to (1);
\coordinate [label=above : $e_2$] (e2) at (0,0);
\draw[-] (2) to [out=120, in=60] (1);
\coordinate [label=above : $e_3$] (e3) at (0,.6);

\end{tikzpicture}

\caption{A generator of the cyclic group $\mathbb{Z}/6\mathbb{Z}$ acts by interchanging the two vertices while cyclically permuting the three edges. The symmetric group 
$\mathfrak{S}_3=\left<\tau,\sigma \ | \ \tau^2=\sigma^3=1, \sigma\tau=\tau\sigma^{-1}\right>$ 
acts as follows: $\sigma$ cyclically permutes the three edges, and $\tau$ interchanges the two vertices, flipping $e_1$ while interchanging $e_2$ and $e_3$. }
\label{genus 2}

\end{figure}

\subsection{Harmonic Galois Theory}\label{GalTheory}

In \cite{HarmGal}, we constructed a harmonic Galois theory for finite graphs with the goal of answering the following general question: if we fix a connected base graph $X$, how can we classify the connected harmonic $G$-covers $\phi:Y\rightarrow X$? By a harmonic $G$-cover $\phi:Y\rightarrow X$, we mean a harmonic group action $(G,Y)$ together with an isomorphism $\overline{\phi}:G\backslash Y\tilde{\rightarrow}X$.  Composing the isomorphism $\overline{\phi}$ with the quotient morphism $\phi_G$ then yields a harmonic morphism  $\phi:=\overline{\phi}\circ\phi_G$ from $Y$ to $X$. In order to explain the classification, we need to introduce the following definitions, which are motivated by the Galois theory of algebraic curves defined over non-algebraically closed fields.

\begin{definition}\label{DI}
Suppose that $\phi:Y\rightarrow X$ is a harmonic $G$-cover of $X$, and $y\in V(Y)$ has image $x:=\phi(y)$. The \emph{decomposition group} $\Delta_y\le G$ at $y$ is the stabilizer of the connected component of the fiber $Y_x:=\phi^{-1}(x)$ containing $y$. The \emph{inertia group} $I_y$ at $y$ is the stabilizer subgroup of $y$ in $G$. Note that $I_y\le \Delta_y$, and the decomposition / inertia groups form conjugacy classes in $G$ as $y$ varies over the fiber $Y_{x}$. We say that $\phi$ is \emph{horizontally unramified} or \emph{\'etale at $y$} if $I_y=\{\varepsilon\}$, and the cover $\varphi$ is \emph{horizontally unramified} if it is horizontally unramified at all $y\in V(Y)$. The \emph{horizontal ramification index} at $y$ is $m_y:=|I_y|$, and the \emph{inertia degree} at $y$ is $f_y:=\#\{\textrm{vertices in the connected component of $Y_x$ containing $y$}\}$. Finally, the \emph{vertical multiplicity} at $y$ is $v_y:=\#\{\textrm{$\phi$-vertical edges incident to $y$}\}$. Since all edges of the fiber $Y_x$ are $\phi$-vertical, $v_y$ is the degree of the vertex $y$ in the graph $Y_x$. Since by Proposition~\ref{Crit} the inertia group $I_y$ acts freely on the edges incident to $y$, we see that the vertical multiplicity satisfies $v_y=m_yw_y$ for some $w_y\ge 0$. Note that the numbers $m_y$, $f_y$, $v_y$, and $w_y$  are independent of the vertex $y\in Y_x$, and only depend on the image vertex $x=\phi(y)$. The \emph{branch locus} of $\phi$ is the set of vertices $B\subset V(X)$ for which the corresponding fibers have either $m>1$ or $v>0$. 
\end{definition}

As evidence that these definitions are good analogues of their algebro-geometric / number-theoretic counterparts, we prove the following graph-theoretic version of the Fundamental Identity for primes in Galois extensions of global fields (see e.g. \cite{Neu} Prop. 8.2).

\begin{prop}
Suppose that $\phi:Y\rightarrow X$ is a harmonic $G$-cover and $y\in V(Y)$ with $x:=\phi(y)$. Let $n$ be the number of connected components of the fiber $Y_x$, and $m,f$ be the ramification index and inertia degree at points of the fiber respectively. Then $\deg(\phi)=mfn$.
\end{prop}

\begin{proof}
We have the equalities $\deg(\phi)=|G|=|G/\Delta_y||\Delta_y/I_y||I_y|$. Since $G$ acts transitively on the set of connected components of $Y_x$, the orbit stabilizer theorem yields $n=|G/\Delta_y|$. Similarly, since $\Delta_y$ acts transitively on the vertices of the connected component of $Y_x$ containing $y$, we see that $f=f_y=|\Delta_y/I_y|$. Putting these observations together with the definition $m=m_y:=|I_y|$ yields the Fundamental Identity.
\end{proof}

\noindent
This Fundamental Identity for graphs provides further justification for our proposal in section 2 of \cite{HarmGal} to interpret $\phi$-vertical edges  not as ``vertical ramification'' as in \cite{BN}, but rather as the graph-theoretic analogue of an extension of residue fields. 

An important tool for our study of maximal graph groups is the graph-analogue of the Riemann-Hurwitz formula established in \cite{BN}, which we state here in the special case of harmonic $G$-covers as reformulated in section 2 of \cite{GenusBnds}:
\begin{prop}[\cite{BN} Theorem 2.14]
Suppose that $\phi:Y\rightarrow X$ is a connected harmonic $G$-cover. Then
$$
2g(Y)-2=|G|(2g(X)-2+R),
$$
where the ramification number $R:=\sum_{x\in V(X)}[2(1-\frac{1}{m_x})+w_x]$. Here $m_x:=m_y$ and $w_x:=w_y=\frac{v_y}{m_y}$ for any choice of $y\in V(Y_x)$.
\end{prop}
\noindent
In section 5 of \cite{GenusBnds}, we used the graph-theoretic Riemann-Hurwitz formula to show that if $(G,Y)$ is a maximal harmonic $G$-action, then the quotient $G\backslash Y$ is a tree, and the ramification number for the quotient morphism $Y\rightarrow G\backslash Y$ is $R=\frac{7}{3}$. Moreover, an earlier proposition from \cite{GenusBnds} shows that $R=\frac{7}{3}$  can only occur in three ways:
\begin{prop}[\cite{GenusBnds} Prop. 3.3 and section 5]\label{BranchLoc}
Suppose that $\phi:Y\rightarrow X$ is a connected harmonic $G$-cover. Then $\phi$ is maximal ($|G|=6(g(Y)-1)$) if and only if  $X$ is a tree and the ramification number for $\phi$ is $R=\frac{7}{3}$. In this case, there are exactly three possibilities for the branch locus $B\subset V(X)$, up to a reordering of the branch points:
\begin{enumerate}
\item[(i)] a single branch point with $m=3, w=1$;
\item[(ii)] two branch points with ramification vector $(m_1, m_2; w_1, w_2)=(3, 2; 0,0)$;
\item[(iii)] two branch points with ramification vector $(m_1, m_2; w_1, w_2)=(3, 1; 0,1)$.
\end{enumerate}
\end{prop}

In section 3 of \cite{HarmGal}, we showed that horizontally unramified $G$-covers of $X$ are classified by a certain group (called the \emph{\'etale fundamental group of $X$}) which is isomorphic to the free profinite completion of the free group on countably many generators. In section 4 of \emph{loc. cit.} we gave a more concrete description of this result. Since we will only need to use this description in the case where $X$ is a tree, we content ourselves with describing that case here: to give an unflipped horizontally unramified $G$-cover of a tree $X$, we just need to specify, for each vertex of $X$, a finite, symmetric, and unordered multi-set of non-trivial elements of $G$. By a multi-set, we mean that the elements of $G$ may appear with multiplicity, and by symmetric we mean that if the element $\rho$ appears, then $\rho^{-1}$ also appears with the same multiplicity. If $S$ is such a multi-set, we may construct the Cayley graph $\textrm{Cay}(G,S)$ with vertex set $G$ as follows (see \cite{HarmGal}, Example 2.8): for each vertex $g\in G$ and for each $\rho\in S$, there is an edge from $g$ to $g\rho$. Furthermore, if $\rho\ne\rho^{-1}$, then the $\rho$-edge from $g$ to $g\rho$ is identified with the $\rho^{-1}$-edge from $g\rho$ to $g=g\rho\rho^{-1}$. Edges coming from involutions in $S$ are not identified in this fashion. Inverse-pairs of group elements appearing with multiplicity in $S$ yield multiple edges of $\textrm{Cay}(G,S)$. The resulting Cayley graph supports a natural unflipped harmonic $G$-action given by left-multiplication on the vertex labels in $G$. We associate to each vertex of $X$ the Cayley graph constructed from the given multi-set; these form the fibers of the corresponding unflipped horizontally unramified $G$-cover, and they are glued together according to the tree $X$. The union of the multi-sets must generate the group $G$ in order for the resulting $G$-cover to be connected. 

We illustrate this construction for the symmetric group $G=\mathfrak{S}_3$ and $X$ the graph consisting of two vertices $x_1$ and $x_2$ connected by a single edge $e$. We have the presentation
$$
\mathfrak{S}_3=\left<\tau,\sigma \ | \ \tau^2=\sigma^3=1, \sigma\tau=\tau\sigma^{-1}\right>.
$$
Choose $S_1=\{\sigma,\sigma^{-1}\}$ and $S_2=\{\tau\}$ for the symmetric multi-sets corresponding to $x_1$ and $x_2$ respectively. Their union generates $\mathfrak{S}_3$, so the corresponding horizontally unramified $G$-cover $Y^{\textrm{ur}}$ will be connected. The fiber $Y^{\textrm{ur}}_{x_i}$ over the vertex $x_i$ is given by the (disconnected) Cayley graph $\textrm{Cay}(\mathfrak{S}_3,S_i)$, and vertices labeled by the same group element in the two fibers are connected by an edge lying over $e$. The group $\mathfrak{S}_3$ acts harmonically on $Y^{\textrm{ur}}$ via left-multiplication on the group elements labeling the vertices.

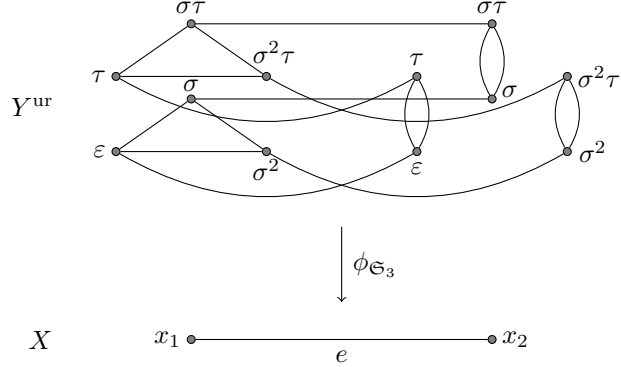
\begin{figure}[h]
\centering
\begin{tikzpicture}
\tikzstyle{every node}=[circle, draw, fill=black!50, inner sep=0pt, minimum width=3pt]

\node (1) at  (-1,1)  {};
\coordinate [label=left : $\tau$] (tau) at (-1.1,1);
\node (2) at  (0,1.7) {} ;  
\coordinate [label=above : $\sigma\tau$] (sigmatau) at (0,1.7);      
\node (3) at  (1,1) {};
\coordinate [label=above : $\sigma^2\tau$] (sigma2tau) at (1.1,1);

\node (4) at  (-1,0)  {};
\coordinate [label=left : $\varepsilon$] (epsilon) at (-1.1,0);
\node (5) at  (0,.7) {} ;  
\coordinate [label=above : $\sigma$] (sigma) at (0,.75);      
\node (6) at  (1,0) {};
\coordinate [label=below : $\sigma^2$] (sigma2) at (1,0);

\draw[-] (1) to (2);
\draw[-] (2) to (3);
\draw [-] (3) to (1);

\draw[-] (4) to (5);
\draw[-] (5) to (6);
\draw[-] (6) to (4);

\node (7) at  (3,1)  {};
\coordinate [label=above : $\tau$] (tau) at (3,1.1);
\node (8) at  (4,1.7) {} ;  
\coordinate [label=above : $\sigma\tau$] (sigmatau) at (4,1.7);      
\node (9) at  (5,1) {};
\coordinate [label=right : $\sigma^2\tau$] (sigma2tau) at (5.1,1);

\node (10) at  (3,0)  {};
\coordinate [label=below : $\varepsilon$] (epsilon) at (3,-.1);
\node (11) at  (4,.7) {} ;  
\coordinate [label=right : $\sigma$] (sigma) at (4.1,.8);   
\node (12) at  (5,0) {};
\coordinate [label=right : $\sigma^2$] (sigma2) at (5.1,0);

\draw[-] (7) to [out=240, in=120] (10);
\draw[-] (7) to [out=300, in=60] (10);
\draw[-] (8) to [out=240, in=120] (11);
\draw[-] (8) to [out=300, in=60] (11);
\draw[-] (9) to [out=240, in=120] (12);
\draw[-] (9) to [out=300, in=60] (12);

\draw[-] (2) to (8);
\draw[-] (4) to [out=330, in=210] (10);
\draw[-] (6) to [out=330, in=210] (12);
\draw[-] (5) to (11);
\draw[-] (1) to [out=330, in=210] (7);
\draw[-] (3) to [out=330, in=210] (9);

\coordinate [label=left : $Y^{\textrm{ur}}$] (Yet) at (-1.8,.6);

\draw[->] (2,-1) to (2,-2);
\coordinate [label=right : $\phi_{\mathfrak{S}_3}$] (phiS3) at (2.1,-1.5);

\node (13) at (0,-2.5) {};
\coordinate [label=left : $x_1$] (x1) at (-.1,-2.5);
\node(14) at (4,-2.5) {};
\coordinate [label=right : $x_2$] (x2) at (4.1,-2.5);

\coordinate [label=left : $X$] (X) at (-1.8,-2.5);

\draw[-] (13) to (14);
\coordinate [label=below : $e$] (e) at (2,-2.6);

\end{tikzpicture}
\caption{The unflipped horizontally unramified $\mathfrak{S}_3$-cover of $X$ corresponding to the symmetric multi-sets $S_1=\{\sigma,\sigma^{-1}\}$ and $S_2=\{\tau\}$.}
\end{figure}

The harmonic $\mathfrak{S}_3$-cover constructed above has no flipped edges. But the pairs of vertical edges corresponding to $\tau$ in the fiber $Y^{\textrm{ur}}_{x_2}$ may each be identified to a single flipped edge, thereby obtaining a harmonic $\mathfrak{S}_3$-cover $\overline{Y}^{\textrm{ur}}\rightarrow X$ whose unflipped model is $Y^{\textrm{ur}}$ (see Figure~\ref{S3flipped} and \cite{HarmGal} section 2).

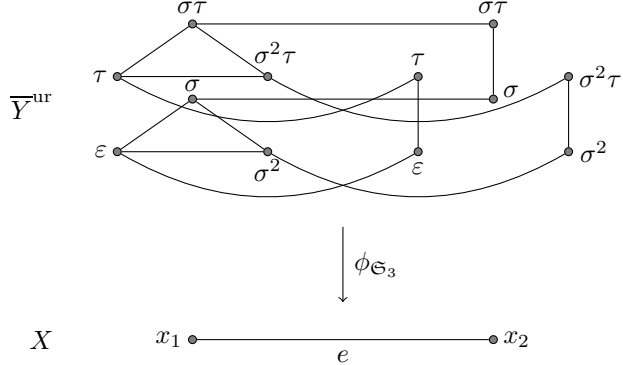
\begin{figure}[h]
\centering
\begin{tikzpicture}
\tikzstyle{every node}=[circle, draw, fill=black!50, inner sep=0pt, minimum width=3pt]

\node (1) at  (-1,1)  {};
\coordinate [label=left : $\tau$] (tau) at (-1.1,1);
\node (2) at  (0,1.7) {} ;  
\coordinate [label=above : $\sigma\tau$] (sigmatau) at (0,1.7);      
\node (3) at  (1,1) {};
\coordinate [label=above : $\sigma^2\tau$] (sigma2tau) at (1.1,1);

\node (4) at  (-1,0)  {};
\coordinate [label=left : $\varepsilon$] (epsilon) at (-1.1,0);
\node (5) at  (0,.7) {} ;  
\coordinate [label=above : $\sigma$] (sigma) at (0,.75);         
\node (6) at  (1,0) {};
\coordinate [label=below : $\sigma^2$] (sigma2) at (1,0);

\draw[-] (1) to (2);
\draw[-] (2) to (3);
\draw [-] (3) to (1);

\draw[-] (4) to (5);
\draw[-] (5) to (6);
\draw[-] (6) to (4);

\node (7) at  (3,1)  {};
\coordinate [label=above : $\tau$] (tau) at (3,1.1);
\node (8) at  (4,1.7) {} ;  
\coordinate [label=above : $\sigma\tau$] (sigmatau) at (4,1.7);      
\node (9) at  (5,1) {};
\coordinate [label=right : $\sigma^2\tau$] (sigma2tau) at (5.1,1);

\node (10) at  (3,0)  {};
\coordinate [label=below : $\varepsilon$] (epsilon) at (3,-.1);
\node (11) at  (4,.7) {} ;  
\coordinate [label=right : $\sigma$] (sigma) at (4.1,.8);      
\node (12) at  (5,0) {};
\coordinate [label=right : $\sigma^2$] (sigma2) at (5.1,0);

\draw[-] (7) to (10);
\draw[-] (8) to (11);
\draw[-] (9) to (12);

\draw[-] (2) to (8);
\draw[-] (4) to [out=330, in=210] (10);
\draw[-] (6) to [out=330, in=210] (12);
\draw[-] (5) to (11);
\draw[-] (1) to [out=330, in=210] (7);
\draw[-] (3) to [out=330, in=210] (9);

\coordinate [label=left : $\overline{Y}^{\textrm{ur}}$] (Yetbar) at (-1.8,.6);

\draw[->] (2,-1) to (2,-2);
\coordinate [label=right : $\phi_{\mathfrak{S}_3}$] (phiS3) at (2.1,-1.5);

\node (13) at (0,-2.5) {};
\coordinate [label=left : $x_1$] (x1) at (-.1,-2.5);
\node(14) at (4,-2.5) {};
\coordinate [label=right : $x_2$] (x2) at (4.1,-2.5);

\coordinate [label=left : $X$] (X) at (-1.8,-2.5);

\draw[-] (13) to (14);
\coordinate [label=below : $e$] (e) at (2,-2.6);

\end{tikzpicture}
\caption{The horizontally unramified $\mathfrak{S}_3$-cover of $X$ with flipped edges corresponding to the symmetric multi-sets $S_1=\{\sigma,\sigma^{-1}\}$ and $S_2=\{\tau\}$.}
\label{S3flipped}
\end{figure}

To allow for horizontal ramification, we introduced in \cite{HarmGal} the notion of a \emph{$G$-inertia structure} on the base $X$, which is simply a collection of subgroups indexed by the vertices of $X$. If $\mathcal{I}=\{I_x \le G \ | \ x\in V(X)\}$ is such a $G$-inertia structure on $X$, then there is a functor $\mathcal{F}^\mathcal{I}$ from the category of horizontally unramified $G$-covers of $X$ to the category of harmonic $G$-covers of $X$ with inertia groups given by the conjugacy classes $C(\mathcal{I}):=\{c(I_x) \ | \ x\in V(X)\}$. The functor acts on each fiber by collapsing the vertex set $G$ of the Cayley graph over $x$ onto the set of left cosets $G/I_x$, removing any loop edges that are produced. In Proposition 5.2 of \cite{HarmGal}, we prove that every harmonic $G$-cover of $X$ with inertia given by $C(\mathcal{I})$ arises via this construction. Thus, every harmonic $G$-cover of $X$ can be described by specifying a $G$-inertia structure on $X$, together with a finite, symmetric, unordered multi-set of non-trivial elements of $G$ for each vertex of $X$. The cover will be connected exactly when $G$ is generated by the union of the inertia groups and the multi-sets.

Returning to our $\mathfrak{S}_3$-example, choose the inertia structure $\mathcal{I}=\{I_1,I_2\}$ with $I_1=\left<\sigma\right>$ and $I_2$ trivial. Then applying the functor $\mathcal{F}^\mathcal{I}$ to the $\mathfrak{S}_3$-cover $\overline{Y}^{\textrm{ur}}\rightarrow X$ has the following effect: the fiber $\overline{Y}^{\textrm{ur}}_{x_2}$ is unchanged, while the fiber $\overline{Y}^{\textrm{ur}}_{x_1}=\textrm{Cay}(\mathfrak{S}_3,S_1)$ is altered by collapsing the vertices onto the two left-cosets of $I_1$ in $\mathfrak{S}_3$ and removing the loop-edges that result (see Figure~\ref{S3examp}).

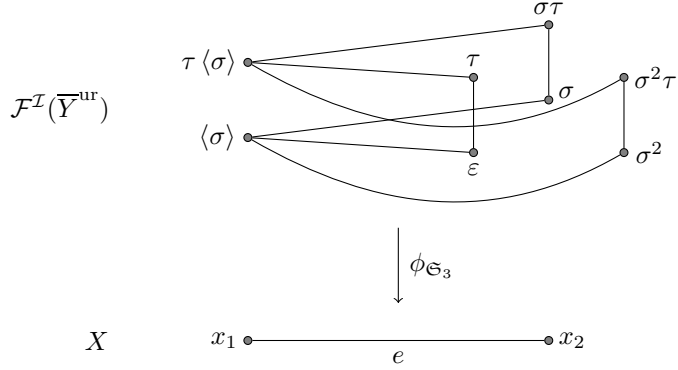
\begin{figure}
\centering
\begin{tikzpicture}
\tikzstyle{every node}=[circle, draw, fill=black!50, inner sep=0pt, minimum width=3pt]

\node (1) at  (0,1.2)  {};
\coordinate [label=left : $\tau\left<\sigma\right>$] (tauI1) at (-.1,1.2);

\node (4) at  (0,.2)  {};
\coordinate [label=left : $\left<\sigma\right>$] (I1) at (-.1,.2);

\node (7) at  (3,1)  {};
\coordinate [label=above : $\tau$] (tau) at (3,1.1);
\node (8) at  (4,1.7) {} ;  
\coordinate [label=above : $\sigma\tau$] (sigmatau) at (4,1.7);      
\node (9) at  (5,1) {};
\coordinate [label=right : $\sigma^2\tau$] (sigma2tau) at (5.1,1);

\node (10) at  (3,0)  {};
\coordinate [label=below : $\varepsilon$] (epsilon) at (3,-.1);
\node (11) at  (4,.7) {} ;  
\coordinate [label=right : $\sigma$] (sigma) at (4.1,.8);      
\node (12) at  (5,0) {};
\coordinate [label=right : $\sigma^2$] (sigma2) at (5.1,0);

\draw[-] (7) to (10);
\draw[-] (8) to (11);
\draw[-] (9) to (12);

\draw[-] (1) to (8);
\draw[-] (4) to  (10);
\draw[-] (4) to [out=330, in=210] (12);
\draw[-] (4) to (11);
\draw[-] (1) to  (7);
\draw[-] (1) to [out=330, in=210] (9);

\coordinate [label=left : $\mathcal{F}^\mathcal{I}(\overline{Y}^{\textrm{ur}})$] (Ybar) at (-1.8,.6);

\draw[->] (2,-1) to (2,-2);
\coordinate [label=right : $\phi_{\mathfrak{S}_3}$] (phiS3) at (2.1,-1.5);

\node (13) at (0,-2.5) {};
\coordinate [label=left : $x_1$] (x1) at (-.1,-2.5);
\node(14) at (4,-2.5) {};
\coordinate [label=right : $x_2$] (x2) at (4.1,-2.5);

\coordinate [label=left : $X$] (X) at (-1.8,-2.5);

\draw[-] (13) to (14);
\coordinate [label=below : $e$] (e) at (2,-2.6);

\end{tikzpicture}
\caption{The harmonic $\mathfrak{S}_3$-cover of $X$ with flipped edges corresponding to the symmetric multi-sets $S_1=\{\sigma,\sigma^{-1}\}, S_2=\{\tau\}$, and $\mathfrak{S}_3$-inertia structure $I_1=\left<\sigma\right>, I_2=\{\epsilon\}$.}
\label{S3examp}
\end{figure}

\section{Proof of Theorem \ref{Main}}\label{Proof}

In this section, we use the results of section~\ref{GalTheory} to prove

\begin{thm1}
A finite group $G$ is a maximal graph group if and only if $|G|\ge 6$ and $G$ is a quotient of the modular group $\Gamma$ with presentation $\Gamma=\left<x,y \ | \ x^2=y^3=1\right>$.
\end{thm1}

\begin{proof}
($\Longleftarrow$) Suppose that $|G|\ge 6$ and $\pi:\Gamma\rightarrow G$ is a surjection. Set $\tau:=\pi(x)$ and $\sigma:=\pi(y)$, so that $\tau$ has order 2 and $\sigma$ has order 3 in $G$. Let $X=\star$ be the point graph. In order to construct a harmonic $G$-cover of $X$, we just need to specify a symmetric multi-set $S$ together with an inertia group $I<G$. For this, we take $S=\left\{\tau\right\}$ and $I=\left<\sigma\right>$, and define $Y:=\mathcal{F}^I(\textrm{Cay}(G,S))$. Then $Y\rightarrow X$ is a harmonic $G$-cover of $X$, with inertia groups given by the conjugacy class of $I$ in $G$. Moreover, $Y$ is connected since $G$ is generated by $I\cup S$. By construction, this is an unflipped action, but each pair of edges corresponding to $\tau$ may be replaced by a single flipped edge to obtain a connected harmonic $G$-cover $\overline{Y}\rightarrow X$. Moreover, since $|I|=3$, every point of $\overline{Y}$ has inertia group of order 3 and is incident to 3 vertical edges. That is, the ramification of $\overline{Y}\rightarrow X$ corresponds to case (i) of Proposition \ref{BranchLoc}: a single branch point $\star$ with $m=3$ and $w=1$. It follows that the ramification number $R=\frac{7}{3}$, so that $|G|=6(g(\overline{Y})-1)$, and $G$ is a maximal graph group.

($\Longrightarrow$) Now suppose that $G$ is a maximal graph group, so there exists a connected harmonic $G$-cover $\overline{Y}\rightarrow X$ where the genus of $\overline{Y}$ satisfies $|G|=6(g(\overline{Y})-1)\ge 6$. Moreover, from \cite{GenusBnds} we know that $X$ is a tree and one of the three branch loci described in Proposition \ref{BranchLoc} occurs. We first consider case (i) of a single branch point with $m=3$ and $w=1$.

From the construction in \cite{HarmGal}, we may assume that $X=\star$ is the point graph, since the part of the tree outside of the single branch point is inessential in this case. Thus, the cover $\overline{Y}\rightarrow X$ is totally degenerate, and its unflipped model $Y$ may be obtained as $\mathcal{F}^I(\textrm{Cay}(G,S))$ for some inertia group $I<G$ and symmetric multi-set $S$ of elements from $G$. Fix such a choice of $I$ and $S$, where we may assume that $I\cap S=\emptyset$. (This is because any edge of the Cayley graph coming from the intersection will be removed as a loop edge when we apply $\mathcal{F}^I$.) Since $m=3$, we must have $I\cong\mathbb{Z}/3\mathbb{Z}$; choose a generator $\sigma$ for this inertia subgroup. The condition $w=1$ means that each vertex of $\overline{Y}$ is incident to $m=3$ vertical edges, and since the $G$-action is totally degenerate, we see that $\overline{Y}$ is in fact 3-regular. This is only possible if every edge of $\overline{Y}$ is flipped, having been obtained from a pair of edges in the unflipped model $Y$. Indeed, since $\textrm{Cay}(G,S)$ has degree at least 2, and the functor $\mathcal{F}^I$ identifies 3 vertices in $\textrm{Cay}(G,S)$ to a single vertex in $Y$, we see that $Y=\mathcal{F}^I(\textrm{Cay}(G,S))$ has degree at least 6 (here we use the fact that $I\cap S=\emptyset$). Hence, $\textrm{Cay}(G,S)$ must be 2-regular, with the property that the edges of $\mathcal{F}^I(\textrm{Cay}(G,S))$ may be identified in pairs to produce $\overline{Y}$. This leads to two possibilities for the multi-set $S$: either $S=\{\tau\}$ where $\tau$ has order 2, or $S=\{\rho,\rho^{-1}\}$ where $\rho\in\tau I$ for some element $\tau$ of order 2. The second option requires some explanation: if the element $\rho$ is to yield a flipped edge of $\overline{Y}$, then it must connect two vertices that are interchanged by an element $\tau\in G$ of order 2. The vertices of $\overline{Y}$ are labeled by the left cosets of $I$, so the edge corresponding to $\rho$ must connect $I$ and $\tau I$. But the element $\rho$ yields an edge of $\overline{Y}$ connecting $I$ to $\rho I$, so it follows that $\rho I=\tau I$, which is equivalent to the stated condition $\rho\in\tau I$.

Since $\overline{Y}$ is connected, we see that $G$ is generated by $I\cup S$. In both of the cases described above, this implies that $G$ is generated by $\tau$ and $\sigma$. Hence, we may define a surjection $\pi:\Gamma\rightarrow G$ by $\pi(x)=\tau$ and $\pi(y)=\sigma$, showing that $G$ is a quotient of $\Gamma$ as required.
	
Now assume that we are in case (ii) or (iii) of Proposition \ref{BranchLoc}: two branch points $x_1$ and $x_2$ with ramification vector $(m_1, m_2; w_1, w_2)=(3, 2; 0, 0)$ or $(3, 1; 0, 1)$. Let $P$ be the unique path between $x_1$ and $x_2$ in the tree $X$. Then as a first simplification, we may assume that $X=P$, since the part of the tree outside of $P$ plays no essential role in the constructions from \cite{HarmGal}. Thus, the $G$-cover $\overline{Y}\rightarrow X$ has an unflipped model $Y$ that corresponds to a pair of symmetric multisets $S_1$, $S_2$ and an inertia structure $\mathcal{I}=\{I_1,\{\varepsilon\},\dots, \{\varepsilon\},I_2\}$ with the identity subgroup at every vertex of valency 2 in $X$. 

In case (ii) the ramification vector is $(3, 2; 0, 0)$, so we must have $I_1\cong\mathbb{Z}/3\mathbb{Z}$ and $I_2\cong\mathbb{Z}/2\mathbb{Z}$; choose generators $\sigma$, $\tau$ of $I_1$ and $I_2$ respectively, so that $\sigma$ has order 3 and $\tau$ has order 2 in $G$. In this case the $G$-cover $Y\rightarrow X$ has no vertical edges, so we may take $S_1=S_2=\emptyset$, which implies (since $Y$ is connected) that $G$ is generated by $I_1\cup I_2$, hence by the generators $\sigma$ and $\tau$. Defining $\pi:\Gamma\rightarrow G$ by $\pi(x)=\tau$ and $\pi(y)=\sigma$ realizes $G$ as a quotient of $\Gamma$ as required.

Finally, we consider case (iii) with ramification vector $(m_1, m_2; w_1, w_2)=(3, 1; 0, 1)$. Then $I_1\cong\mathbb{Z}/3\mathbb{Z}$, but $I_2$ is trivial. As before, choose a generator $\sigma$ of $I_1$, which has order 3 in $G$. Since the fiber over $x_1$ contains no vertical edges, we may take $S_1=\emptyset$. The points of the fiber $\overline{Y}_{x_2}$ have vertical multiplicity $v_2=m_2w_2=1$, which implies that the unflipped model $Y_{x_2}=\textrm{Cay}(G,\{\tau\})$ for some $\tau\in G$ of order 2. Since $Y$ is connected, it follows that $G$ is generated by $I_1\cup\{\tau\}$, hence by $\sigma$ and $\tau$. As in the previous cases, we see that $G$ is a quotient of $\Gamma$. This final case is illustrated for $G=\mathfrak{S}_3$ in Figure~\ref{S3examp}.
\end{proof}

Table~\ref{GroupTable} lists the maximal graph groups that arise in low genus. In particular, $g=6$ is the first genus for which no maximal graph exists, improving the result established in Proposition 9.2 of \cite{GenusBnds} that there is no maximal graph of genus 12. 

\begin{table}[h]\label{GroupTable}
\centering
\begin{tabular}{|c|c|c|}
\hline
Genus $g$ & $6(g-1)$ & Maximal graph groups for genus $g$\\
\hline
\hline
2 & 6   & $\mathbb{Z}/6\mathbb{Z}, \mathfrak{S}_3$\\
\hline
3 & 12 & $\mathfrak{A}_4$\\
\hline
4 & 18 & $\mathfrak{S}_3\times\mathbb{Z}/3\mathbb{Z}$\\
\hline
5 & 24 & $\mathfrak{S}_4, \mathfrak{A}_4\times\mathbb{Z}/2\mathbb{Z}$\\
\hline
6 & 30 & none\\
\hline
\end{tabular}
\caption{Maximal graph groups for low genus}
\label{tab:maxgroups}
\end{table}

Of course, the modular group $\Gamma\cong PSL_2(\mathbb{Z})$ has been studied intensively due to its central role in number theory and geometry, and much is known about its finite quotients. For instance, a 1901 result of G.A. Miller \cite{Mil} says that all alternating and symmetric groups are quotients of $\Gamma$ except for $\mathfrak{A}_6, \mathfrak{A}_7, \mathfrak{A}_8$ and $\mathfrak{S}_5, \mathfrak{S}_6, \mathfrak{S}_8$. Hence, by Theorem~\ref{Main}, all nonabelian alternating and symmetric groups are maximal graph groups, except for Miller's exceptions. In \cite{LS}, Liebeck and Shalov prove that all but finitely many of the finite simple classical groups different from $PSp_4(2^k)$ and $PSp_4(3^k)$ are quotients of $\Gamma$.

The proof of Theorem~\ref{Main} also reveals the following description of maximal graphs.

\begin{prop}\label{DM}
Suppose that $(G,Y)$ is a maximal harmonic group action, and that every vertex of $Y$ has degree at least 2. Then there exists a unique trivalent graph $Y_0$ such that $(G,Y_0)$ is a maximal harmonic group action, and $Y$ is obtained from $Y_0$ by subdividing each edge of $Y_0$ into $m\ge 1$ edges.
 \end{prop}
 
 \begin{proof}
 As in the proof of Theorem~\ref{Main}, the hypothesis that $Y$ has no vertex of degree 1 implies that the quotient $G\backslash Y$ is a path $P_n$ of length $n\ge 0$. If $n=0$, then we are in case (i) of Proposition~\ref{BranchLoc}, and $Y$ is trivalent. If $n>0$ is even, then we are in case (ii), while $n>0$ odd corresponds to case (iii). In both of these cases, there is a branch point $x_1\in P_n$ with horizontal ramification $m_1=3$ and $w_1=0$, so the points in the fiber $Y_{x_1}$ have degree 3 in the graph $Y$. All other vertices of $Y$ have degree 2, and simply removing these vertices produces a trivalent graph $Y_0$ as required.
\end{proof}

Hence, up to subdivision and the contraction of leaves, the maximal graphs are exactly the trivalent graphs that admit a group of automorphisms acting regularly (i.e. simply transitively) on the set of directed edges (1-arcs). This is the class of trivalent symmetric graphs of type 1' studied by Djokovi\'c and Miller in \cite{DM}, where they prove that every such 1-arc regular object is the quotient of a 1-arc regular action of the modular group $\Gamma$ on the infinite trivalent tree. From this point of view, the interest of Theorem~\ref{Main} is that the Djokovi\'c-Miller class of graphs arises naturally from the theory of harmonic group actions, without restricting attention at the outset to trivalent graphs. For more on the theory and classification of trivalent symmetric graphs, see \cite{ConDob}, \cite{ConLor}, and \cite{ConNed}.

\section{Connection with Hurwitz surfaces}\label{Hurwitz}

In this section we briefly summarize the theory of bipartite combinatorial maps in order to illustrate an explicit connection between maximal graphs and Riemann surfaces. We will follow the exposition in \cite{Jones}, adapting sections 1-3 of \textit{loc. cit.} to the specific case of maximal graphs.

Suppose that $(G,Y)$ is a maximal harmonic group action. By Proposition~\ref{DM}, after contracting leaf edges and adding/deleting vertices of degree 2, we may assume that $Y$ is bipartite, having been obtained from a trivalent graph $Y_0$ by subdividing each edge into two. We think of the trivalent vertices as black, and the degree-2 vertices as white. The ramification of the quotient morphism $Y\rightarrow G\backslash Y$ is of type (ii) from Proposition~\ref{BranchLoc}, with no vertical edges, horizontal ramification index 3 at all black vertices, and horizontal ramification index 2 at all white vertices.

Pick a pair of adjacent vertices $b$ (black) and $w$ (white), connected by an edge $e$. Then the inertia subgroup $I_b$ has order 3, while the inertia subgroup $I_w$ has order 2. Pick a generator $\sigma$ for $I_b$, and let $\tau$ be the involution generating $I_w$. Note that the choice of $\sigma$ determines one of the two possible cyclic orderings of the edges adjacent to $b$: $(e,\sigma e,\sigma^2 e)$. Moreover, if $b'$ is any other black vertex, there exists a group element $\gamma\in G$ such that $\gamma b=b'$, and the inertia group at $b'$ is $I_{b'}=\gamma I_b\gamma^{-1}$, generated by $\gamma\sigma \gamma^{-1}$. Hence, the choice of generator $\sigma$ actually determines a cyclic ordering of the edges incident to $b'$ as well: $(e', \gamma\sigma \gamma^{-1}e', \gamma\sigma^2 \gamma^{-1}e')$, where $e':=\gamma e$. Note that the harmonic $G$-action on $Y$ preserves the cyclic orderings thus specified at each of the black vertices.

Let $E$ be the set of edges of $Y$. Then $G$ acts simply transitively on $E$, since $E$ may be identified with the set of directed edges (1-arcs) of the trivalent graph $Y_0$. For each black vertex $b'$, let $c_{b'}$ denote the 3-cycle in $\mathfrak{S}(E)$ corresponding to the cyclic ordering of the edges incident to $b'$. Then define $g_0:=\prod_{b'}c_{b'}$ to be the product of these 3-cycles. Similarly, for each white vertex $w'$, let $c_{w'}$ be the transposition interchanging the two edges incident to $w'$, and define $g_1:=\prod_{w'}c_{w'}$ to be the product. Note that $g_0$ has order 3, and $g_1$ has order 2.

The permutations $g_0$ and $g_1$ define an \emph{oriented bipartite map} $\mathcal{B}$ with underlying graph $Y$, i.e. an embedding of $Y$ into a compact orientable surface $S$ such that each connected component of the complement $S - Y$ is homeomorphic to an open disc. Indeed, the set of black vertices corresponds to the disjoint cycles of $g_0$, while the set of white vertices corresponds to the disjoint cycles of $g_1$. The edges $E$ are determined by the overlaps between the cycles of $g_0$ and $g_1$. Finally, the faces of $\mathcal{B}$ correspond to the disjoint cycles of the product $g_0g_1$, with a cycle of length $k$ corresponding to a face bounded by $2k$ edges.

Let $M:=\left<g_0,g_1\right>$ be the subgroup of $\mathfrak{S}(E)$ generated by $g_0$ and $g_1$, called the \emph{monodromy group} of the map $\mathcal{B}$. Then the automorphism group $\textrm{Aut}_0(\mathcal{B})$ of the (oriented and bipartite) map $\mathcal{B}$ is the centralizer of $M$ in $\mathfrak{S}(E)$. By our construction, the group $G$ is certainly a subgroup of $\textrm{Aut}_0(\mathcal{B})$, and since $G$ already acts simply transitively on $E$, it follows that the map is \emph{bipartite-regular}, and that $\textrm{Aut}_0(\mathcal{B})=G\cong M$ via the map $\sigma\mapsto g_0, \tau\mapsto g_1$.

Since $\mathcal{B}$ is bipartite-regular, it follows that each face contains the same number of edges, $2k$, where $k$ is the common order of $g_0g_1$ in $M$ and $\sigma\tau$ in $G$. Moreover, the total number of faces is $|F|=|E|/k$, because each face corresponds to a cycle of length $k$ in $g_0g_1$. Since the number of vertices is $|V|=|E|/3+|E|/2$, we find that the Euler characteristic of the surface $S$ underlying $\mathcal{B}$ is
$$
2-2g(S)=|V|-|E|+|F|=\frac{|E|}{3}+\frac{|E|}{2}-|E|+\frac{|E|}{k}=|E|(\frac{1}{k}-\frac{1}{6})
$$
Since $|G|=|E|$, we find that
$$
|G|=\frac{12k(g(S)-1)}{k-6}.
$$
For a fixed genus $g(S)\ge 2$, the quantity on the right hand side is maximized for $|\sigma\tau|=k=7$, which occurs exactly when $G$ is a Hurwitz group and $|G|=84(g(S)-1)$. But this construction only produces a topological surface $S$, not a Riemann surface.

In order to endow the surface $S$ with a complex structure, we appeal to the existence of a \emph{universal bipartite map} $\hat{\mathcal{B}}$ on the extended hyperbolic plane $\overline{\mathcal{U}}$, as described in section 3 of \cite{Jones}. Here
$$
\overline{\mathcal{U}}:=\{z=x+iy \ | \ y>0\}\cup\mathbb{Q}\cup\{\infty\}
$$
is the union of the hyperbolic upper half-plane and the rational projective line. The modular group $\Gamma=PSL_2(\mathbb{Z})$ acts on $\overline{\mathcal{U}}$ as orientation-preserving isometries of the hyperbolic geometry. The elements of the modular group are the M\"obius transformations, $T$, defined by
$$
T:z\mapsto \frac{az+b}{cz+d} \qquad a,b,c,d\in\mathbb{Z} \ \textrm{and} \ ad-bc=1.
$$
The black vertices of $\hat{\mathcal{B}}$ are the set of rational numbers with even numerator and odd denominator; the white vertices are the set  of rational numbers with odd numerator and denominator. The edges are the hyperbolic geodesics between vertices $\frac{a}{b}$ and $\frac{c}{d}$ such that $ad-bc=\pm 1$. The group of orientation- and color-preserving automorphisms of this map is the congruence subgroup $\Gamma(2)=\{T\in\Gamma \ | \ b\equiv c\equiv 0 \ (2)\}$, which is a free group of rank 2 generated by the following two M\"obius transformations:
$$
T_0: z\mapsto \frac{z}{-2z+1} \qquad \textrm{and} \qquad T_1:z\mapsto\frac{z-2}{2z-3}.
$$
Hence, given any bipartite map $\mathcal{B}$ defined by permutations $g_0$ and $g_1$ as above, we can define a surjection from $\Gamma(2)$ onto the monodromy group $M$ by $T_0\mapsto g_0, T_1\mapsto g_1$. Thus, $\Gamma(2)$ acts transitively on the edges $E$ of the map $\mathcal{B}$. Denoting by $B<\Gamma(2)$ the stabilizer of a chosen edge $e\in E$, we obtain an isomorphism of oriented bipartite maps $B\backslash\hat{\mathcal{B}}\cong \mathcal{B}$. Moreover, $\mathcal{B}$ is bipartite-regular if and only if $B$ is normal in $\Gamma(2)$, in which case $\textrm{Aut}_0(\mathcal{B})\cong\Gamma(2)/B$. Since the surface $S$ underlying the map $\mathcal{B}$ is the quotient of the hyperbolic plane by a finite-index subgroup, $B$, of the modular group, it follows that $S$ has the structure of a Riemann surface. In particular, if the bipartite map $\mathcal{B}$ comes from a maximal harmonic group action $(G,Y)$, then $G=\textrm{Aut}_0(\mathcal{B})$ acts as a group of holomorphic automorphisms of the Riemann surface $S$. If $G$ is actually a Hurwitz group, then as shown above, $|G|=84(g(S)-1)$, and $S$ is actually a Hurwitz surface.

Of course, there are other constructions that also yield Riemann surfaces from finite graphs with extra structure. For instance, 
in \cite{BM}, Brooks and Makover produce compact Riemann surfaces from trivalent graphs together with a cyclic ordering of the edges at each vertex. Their construction proceeds by gluing together ideal hyperbolic triangles according to instructions specified by the graph. Applying this construction to a maximal harmonic group action $(G,Y_0)$ yields the Riemann surface constructed above via the associated bipartite map, but now endowed with a dual triangulation. While these constructions are well-known, it is interesting to observe how the study of harmonic group actions, which has its origin in the recently developed graph-theoretic Riemann-Roch and Abel-Jacobi theory \cite{BNRR}, \cite{BN}, leads naturally to the well-established topics of trivalent symmetric graphs and combinatorial regular maps. 

\section*{Acknowledgements}
I would like to thank my student Gus Black for many conversations about this research while he was an undergraduate at Lawrence University, and for determining the maximal graph groups that arise for graphs of low genus in Table~\ref{GroupTable}.

\bibliographystyle{amsplain}

\end{document}